\documentclass[10pt]{amsart}
\usepackage{amssymb,amsmath}
\usepackage[usenames]{color}

\newtheorem{thm}{Theorem}[section]

\newtheorem{lem}[thm]{Lemma}

\theoremstyle{definition}

\theoremstyle{remark}

\def\cC{\mathcal C}

\begin{document}

%-------------------------------------------------------
\title[Configurations in the Euclidean plane associated to a system of equations]{Configurations in the Euclidean plane associated to a system of equations}

\author{Francesco Colangelo}

\address{
University of Basilicata\\
Italy}

\email{francesco.colangelo@unibas.it}

\date{}

\begin{abstract}
In the Euclidean plane ${\bf{E}}^2$, fix four pairwise distinct points
\begin{equation*}
%\label{eqA}
\begin{array}{ccc}
A=(a_1,a_2),\ B=(b_1,b_2),\ C=(c_1,c_2),\ D=(d_1,d_2),
\end{array}
\end{equation*}
together with four non-zero real numbers $k_A,k_B,k_C,k_D$. We show that System (*) consisting of the following four equations
in the unknowns $X=(x_1,x_2)$ and $Y=(y_1,y_2)$
\begin{equation*}
%\label{egy}
\frac{1}{\|X-T\|^2}  +\frac{1}{\|Y-T\|^2}=k_T, \quad T\in\{A,B,C,D\}
\end{equation*}
has finitely many solutions $(X,Y)$ (counting also those with complex coordinates) provided that both of the following two conditions are satisfied:
\begin{itemize}
\item[($i$)] no three of the fixed points $A,B,C,D$ are coplanar;
\item[($ii$)] no three of the four circles of center $T$ and radius $1/\sqrt{k_T}$ with
share a common point in ${\bf{E}}^2$.
\end{itemize}
Furthermore, we exhibit a configuration $ABCD$ showing that System (*) satisfying $(i)$ and $(ii)$ may have many real solutions $(X,Y)$.  
This result is the planar version of an analog problem in the Euclidean space arising from applications to genetics, investigated in the recent papers \cite{cif} and \cite{ak2024}.
\end{abstract}

\maketitle
\section{Introduction}
Problems in higher dimensional Euclidean spaces often ask for the solution of a system of equations of the form $f_1(x_1,x_2\ldots,x_n)=0, \ldots, f_k(x_1,x_2\ldots,x_n)=0$ where the $f_i$'s are real functions depending on Euclidean distances in ${\mathbf{E}}^n$; see, for instance, \cite{Lib}. In most cases, polynomial systems occur in this way, i.e. the functions $f_i$ are polynomial with real coefficients. Nevertheless, systems with rational functions $f_i$ may also be of interest, especially in connection with applications to other research areas. In \cite{cif},  it was pointed out that, for the particular case of $n=3$, the following System (\ref{egy2}) arises from the study of inferring the 3D DNA structure of the diploid organism from partially phased data: In ${\mathbf{E}}^{n}$, for a fixed set $\mathcal{T}$ of $2n$ pairwise distinct points together with  a set $\{k_A,k_B,\ldots\}$ of $2n$ non-zero real numbers, System (\ref{egy2}) is in the unknowns $X$, $Y$ with $X,Y\in {\mathbf{E}}^n$ and consists of the $2n$ equations 
\begin{equation}
\label{egy2}\frac{1}{\|X-T\|^2}  +\frac{1}{\|Y-T\|^2}=k_T, \quad T\in \mathcal{T}
\end{equation}
% Since the above system has at least two (trivial) solutions, namely $(X^*,Y^*)$ and $(Y^*,X^*)$, the question %arises whether some more solutions might exist. 
Actually, System (\ref{egy2}) may have infinitely many solutions for some configurations of points. This was pointed out for $n=3$; see \cite{ak2024}. On the other hand, a ``sufficiently generic'' choice of $\mathcal{T}$ appears to ensure the finiteness of the solutions for System (\ref{egy2}). This was shown in \cite{cif} for $n=3$. A more precise result was proven in \cite{ak2024}, namely System (\ref{egy2}) for $n=3$ has finitely many solutions provided that the configuration of the points in  $\mathcal{T}=\{A,B,C,D,E,F\}$  %and the ubiquity of the points $X^*,Y^*$ with respect to %$\mathcal{T}$ 
satisfies the following two conditions:
\begin{itemize}
\item[(I)] no four of the fixed points $A,B,C,D,E,F$ are coplanar;
\item[(II)] no four of the six spheres of center $T\in\{A,B,C,D,E,F\}$ and radius $1/\sqrt{k_T}$ with
%\begin{equation}
%\label{kxy}
%k_T=\frac{1}{\|X^*-T\|^2} +\frac{1}{\|Y^*-T\|^2}
%\end{equation}
share a common point in ${\bf{E}}^3$.
\end{itemize}
In this paper we investigate the case $n=2$. Our main result is the following theorem which appears to be an analog of the above quoted result for $n=3$.  
\begin{thm}
\label{main}
In the Euclidean plane ${\bf{E}}^2$, fix four pairwise distinct points
\begin{equation*}
\label{eqA}
\begin{array}{ccc}
A=(a_1,a_2),\, B=(b_1,b_2),\, C=(c_1,c_2)\, D=(d_1,d_2)
\end{array}
\end{equation*}
together with four non-zero real numbers, $k_A,k_B,k_C,k_D$. Then System (*) consisting of the following four equations
in the unknowns $X=(x_1,x_2)$ and $Y=(y_1,y_2)$
\begin{equation}
\label{egy}
\frac{1}{\|X-T\|^2}  +\frac{1}{\|Y-T\|^2}=k_T, \quad T\in\{A,B,C,D\}
\end{equation}
has finitely many solutions $($counting also those with complex coordinates$)$ provided that both of the following two conditions are satisfied:
\begin{itemize}
\item[$(i)$] no three of the fixed points $A,B,C,D$ are coplanar;
\item[$(ii)$] no three of the four circles of center $T$ and radius $1/\sqrt{k_T}$ with
share a common point in ${\bf{E}}^2$.
\end{itemize}
\end{thm}
In Section \ref{exampl}  we exhibit a configuration $ABCD$ together with four real numbers $k_A,k_B,k_C,k_D$ for which System (*) satisfies $(i)$ and $(ii)$ and has several real solutions $(X,Y)$. 
We also exhibit a System (*) with infinitely many solutions such that $(i)$ is not satisfied, but $(ii)$ it is; see Section  \ref{seex2}.  
\section{Backgrounds on plane algebraic curves}
Let $\cC$ be an irreducible, complex, plane curve of degrre $\geq 2$. For the study of $\cC$, we rely on concepts and results concerning its branches (places), and the field $\mathbb{C}((\tau))$ of formal power series $\alpha(\tau)=\sum\limits_{i\ge \ell}^{\infty}u_i\tau^i$ in an indeterminate $\tau$. Here, $\mathbb{C}((\tau))$ is the rational field of the ring $\mathbb{C}[[\tau]]$ of all formal power series $\alpha(\tau)=\sum\limits_{i\ge 0}^{\infty}u_i\tau^i$.
Our notation and terminology are standard; see \cite{lefschetz2012algebraic,shafaverich2013basic,seidenberg1968}. In particular, $\alpha=0$ if and only if $u_i=0$ for every $i$, and for $\alpha\neq 0$, ${\rm{ord}}(\alpha(\tau))$ denotes the smallest $i$ with $u_i\neq 0$.

Every point $P\in \cC$ is the center of at least one branch of $\cC$. If $P$ is a non-singular point of $\cC$, then exactly one branch of $\cC$ is centered at $P$. For a singular point of $\cC$ there may be more then one but finitely many branches centered at that point. If $P=(\xi,\eta)=(\xi:\eta:1)\in\cC$ is an affine point, a primitive representation of a branch $\gamma$ of $\cC$ centered at $P$ is a pair $(x(\tau),y(\tau))$ so that 
\begin{equation}
\label{eqA20apr}
x(\tau)=\xi+\varphi_1(\tau),\,\quad y(\tau)=\eta+\varphi_2(\tau),
\end{equation}
with $\varphi_i(\tau)\in \mathbb{C}[[\tau]]$ where either  $\varphi_1(\tau)$, or $\varphi_2(\tau)$ is not the zero of $\mathbb{C}[[\tau]])$, i.e. it is not zero identically.

Let $P=(\eta_1:\eta_2:0)$ be a point of $\cC$ at infinity. A primitive branch representation gives a branch of $\cC$ centered at $P$, say 
$(x(\tau):y(\tau):1)$ with $x(\tau),y(\tau)\in \mathbb{C}((\tau))$ of the form
$$x(\tau)=\frac{\eta_1+\varphi_1(\tau)}{\alpha(\tau)},\quad y(\tau)=\frac{\eta_2+\varphi_2(\tau)}{\alpha(\tau)},$$  where $\varphi_i(\tau)\in \mathbb{C}[[\tau]]$, $\alpha(\tau)=t^m \beta(\tau)$ with $m\ge 1$ and $\min_i\{{\rm{ord}}(\varphi_i(\tau))\}<m$.

A branch representation is primitive (or irreducible) if the ${\rm{g.c.d.}}$ of all exponents of $\tau$ in the power series $\varphi_i(\tau)$ is equal to $1$.

Let $M(Y)$ be a polynomial in $Y=(y_1,y_2)$. Then $M(Y)$ is said to vanish at a branch of $\cC$ given by a primitive representation $(x(\tau),y(\tau))$ if $M(x(\tau),y(\tau))$ is the zero power series. If $M(Y)$ vanishes at some branch of $\cC$, it vanishes at every branch of $\cC$ as $\cC$ is irreducible.
Let $L(Y)=L_1(Y)/L_2(Y)$ be a complex rational function with polynomials $L_1(Y),\ L_2(Y)$ which do not vanishes at any branch of $\cC$.
For a branch $\gamma$  of $\cC$ given by a primitive representation $(x(\tau),y(\tau))$, let $L_j(\tau)=L_j(x(\tau),y(\tau))$ with $j=1,2$. Then $L_j(\tau)\in \mathbb{C}((\tau))$, and $\gamma$ is a pole of $L(Y)$ if ${\rm{ord}}(L_1(\tau))<{\rm{ord}}(L_2(\tau))$, and it is a zero of $L(Y)$ if
${\rm{ord}}(L_1(\tau))> {\rm{ord}}(L_2(\tau))$.
By a classical result, see \cite[Theorem 14.1]{seidenberg1968}, every complex rational function $L$
has only finitely many poles, and $L$ has some
poles, unless $L$ is a (non-zero) constant.

%\subsubsection{Third Example} We show a System (*) with infinitely many solutions that does not satisfy Condition (i) in Theorem \ref{main}. Take A=(0,0),\,B=(0,-1),\,C=(0,1),\,D=(1,0), together with X^*=\Big(-\sqrt{\frac{2}{11}(5-\sqrt{14})},0\Big),\quad Y^*=\Big(\sqrt{\frac{2}{11}(5-%\sqrt{14})},0\Big).
%\label{thex}

\section{Outline of the proof} For any solution $(x_1,x_2,y_1,y_2)$ of System (*), we let $X_P=(x_1,x_2)$ and $Y_P=(y_1,y_2)$, and say that $(X_P,Y_P)$ is a solution of System (*).

For $T\in \{A,B,C,D\}$ and indeterminate $Y=(y_1,y_2)$, %with $k_T\|Y-T\|^2\ne 1$,
we introduce the rational function
\begin{equation}
\label{eq241} \Gamma_T(Y)=\frac{\|Y-T\|^2}{k_T\|Y-T\|^2-1}-\|T\|^2.
\end{equation}
With this notation,  System (*) can be rewritten as
\begin{equation}
\label{eq242} \|X-T\|^2=\Gamma_T(Y)+\|T\|^2, \quad T\in\{A,B,C,D\}.
\end{equation}
This shows that if $(X_P,Y_P)$ is a real solution of System (*) then $X_P$ is the common point of the four circles with centers $T$ and radius $\Gamma_T(Y)+\|T\|^2$. Furthermore, if $(X_P,Y_P)$ and
$(X_Q,Y_Q)$ are two real solutions of System (*) such that $X_P\neq X_Q$ but $\Gamma_T(Y_P)=\Gamma_T(Y_Q)$ for at least three $T$'s, say $T\in \{A,B,C\}$ then $A,B,C$ are collinear. This follows from the fact that, in this case, the three circles of equation $\|X-T\|^2=\Gamma_T(Y_P)+\|T\|^2$  with $T\in\{A,B,C\}$ have two distinct common points, namely $X_P$ and $X_Q$, and hence their centers are collinear. Therefore, an approach to the proof of Theorem \ref{main} can rely upon the study of $\Gamma_T(Y)$. We show that this approach, which was introduced and worked out in \cite{ak2024}, is also adequate for our case. In fact, if System (*) is supposed to have infinitely many solutions, then we are able to prove that $\Gamma_T(Y)$ is constant for infinitely many $Y_P$ such that $(X_P,Y_P)$ is a solution of System (*) for some $X_P$, provided that both $(i)$ and $(ii)$ hold. Therefore, the finiteness of the solutions of System (*) holds under our two assumptions $(i)$ and $(ii)$, as claimed in Theorem \ref{main}.

%Our final remark is that assumption $(i)$ cannot be dropped as the following case shows; see Section~\ref{exse}.

\section{Equations linking $X_P$ and $Y_P$}
 \begin{lem}
\label{lemA18apr} Every solution $(X_P,Y_P)$ of System (*) is uniquely determined by $Y_P$.
\end{lem}
\begin{proof}
Subtract the first equation from each of the other three. Then the following linear system in the coordinates of $X$ is obtained:
\begin{equation}
\label{eq243} 2(A-T)\cdot X=\Gamma_T(Y)-\Gamma_A(Y), \quad T\in\{B,C,D\},
\end{equation}
where
$(A-T)\cdot X=2(a_1-t_1)x_1+2(a_2-t_2)x_2+2(a_3-t_3)x_3.$

Since $A,B,C$ are not collinear, the determinant
$$\Delta=\left|
  \begin{array}{cccc}
    a_1-b_1 & a_2-b_2   \\
    a_1-c_1 & a_2-c_2   \\
   \end{array}
\right|
$$
is not zero. Therefore, the classical Cramer's rule applies to the linear system
\begin{equation}
\label{eqB071224}
  \begin{array}{llll}
    2(a_1-b_1)x_1+ 2(a_2-b_2)x_2 = \Gamma_B(Y_P)-\Gamma_A(Y_P);\\
    2(a_1-c_1)x_1+ 2(a_2-c_2)x_2 =\Gamma_C(Y_P)-\Gamma_A(Y_P);\\
  \end{array}
\end{equation}
and it provides a formula for the coordinates of $X_P$ in functions $Y_P$ by using the functions $\Gamma_T(Y)$ with $T\in\{A,B,C,D\}$.
Thus, if $X_P=(x_1,x_2)$ then for $Y=Y_P$
\begin{equation}
\label{eqD071224}
x_1=\frac{1}{2\Delta} \left|\begin{array}{cccc}
\Gamma_B(Y)-\Gamma_A(Y)  & a_2-b_2  \\
\Gamma_C(Y)-\Gamma_A(Y) & a_2-c_2
  \end{array}\right|,
\quad
x_2=\frac{1}{2\Delta} \left|\begin{array}{ccccc}
a_1-b_1 &\Gamma_B(Y)-\Gamma_A(Y) \\
a_1-c_1 & \Gamma_C(Y)-\Gamma_A(Y)
  \end{array}\right|,
\end{equation}
whence the claim follows.
\end{proof}
\begin{lem} Let $(X_P,Y_P)$ be a solution of System (*). For $Y=Y_P$
\label{lemB071224}
\begin{equation}
\label{eqE071224}
\left| \begin{array}{cccccc}
  \Gamma_A(Y) & \Gamma_B(Y)& \Gamma_C(Y) & \Gamma_D(Y) \\
  a_1 & b_1 & c_1 & d_1  \\
  a_2 & b_2 & c_2 & d_2  \\
  1 & 1 & 1 & 1
\end{array}\right|=0.
\end{equation}
\end{lem}
\begin{proof}
From the proof of Lemma \ref{lemA18apr},
\begin{equation}
\label{eqC071224}
  \begin{array}{llll}
    2(a_1-b_1)x_1+ 2(a_2-b_2)x_2 = \Gamma_B(Y_P)-\Gamma_A(Y_P);\\
    2(a_1-c_1)x_1+ 2(a_2-c_2)x_2 =\Gamma_C(Y_P)-\Gamma_A(Y_P);\\
    2(a_1-d_1)x_1+ 2(a_2-d_2)x_2 =\Gamma_D(Y_P)-\Gamma_A(Y_P).
  \end{array}
\end{equation}
Therefore, the rows in the linear system (\ref{eqC071224}) are linearly dependent. This proves the claim.
\end{proof}
\begin{lem}Let $(X_P,Y_P)$ be a solution of System (*). For $Y=Y_P$
\label{lemC071224}\begin{equation}
\label{eqG071224}
\begin{array}{llll}
\vspace{0.2cm}
\Big(\left|\begin{array}{lllll}
\Gamma_B(Y)-\Gamma_A(Y)  & a_2-b_2 \\
\Gamma_C(Y)-\Gamma_A(Y) & a_2-c_2  \\
  \end{array}\right|
  -2\Delta a_1\Big)^2+
  \vspace{0.2cm}
\Big(\left|
\begin{array}{lllll}
a_1-b_1 &\Gamma_B(Y)-\Gamma_A(Y)\\
a_1-c_1 & \Gamma_C(Y)-\Gamma_A(Y) \\
\end{array}\right|
  -2\Delta a_2\Big)^2-\\
\vspace{0.2cm}
  4\Delta^2(\Gamma_A(Y)+\|A\|^2)=0.
\end{array}
\end{equation}
\end{lem}

\begin{proof} The claim follows from (\ref{eq241}) for $T=A$ together with (\ref{eqD071224}).
\end{proof}
%\begin{lem}
%\label{lem091224} For $T\in\{A,B,C,D\}$ 
%$$\Gamma_A$$
%\end{lem}
\section{The variety associated to System (*)}
The solutions $(x_1,x_2,y_1,y_2)$ of System (*) satisfy the system of polynomial equations
\begin{equation}
\label{egybis}
k_T \|X-T\|^2\|Y-T\|^2-\|X-T\|^2-\|Y-T\|^2=0,\quad T\in\{A,B,C,D\}.
\end{equation}
Expanding (\ref{egybis}) gives
\begin{equation}
\label{egybisA}
\begin{array}{l}
     k_T (X_1^2+X_2^2-2t_1X_1-2t_2X_2+t_1^2+t_2^2)(Y_1^2+Y_2^2-2t_1Y_1-2t_2Y_2+t_1^2+t_2^2)-  \\
     (X_1^2+X_2^2-2t_1X_1-2t_2X_2+t_1^2+t_2^2)-(Y_1^2+Y_2^2-2t_1Y_1-2t_2Y_2+t_1^2+t_2^2)=0,\  T\in\{A,B,C,D\}.
\end{array}
\end{equation}

The points $P=(x_1,x_2,y_1,y_2)$ arising from the solutions of System (*) belong to the affine set (i.e. possible singular and/or reducible affine variety) $\mathfrak{V}$ in $\mathbb{R}^4$ and viewed as a variety in $\mathbb{C}^4$. The projective closure of $\mathfrak{V}$ with homogeneous coordinates $(X_1:X_2:Y_1:Y_2:Z)$ is the (possible singular and/or reducible) projective variety $\bar{\mathfrak{V}}$ of equation
\begin{equation}
\label{egybisAA}
\begin{array}{lll}
k_T (X_1^2+X_2^2-2t_1X_1Z-2t_2X_2Z+t_1^2Z^2+t_2^2Z^2)(Y_1^2+Y_2^2-2t_1Y_1Z-2t_2Y_2Z+t_1^2Z+t_2^2)Z-\\
\big((X_1^2+X_2^2-2t_1X_1Z-2t_2X_2Z+t_1^2Z+t_2^2Z)-(Y_1^2+Y_2^2-2t_1Y_1Z-2t_2Y_2Z+t_1^2Z+t_2^2Z))Z^2=0,
\end{array}
\end{equation}
where $T$ runs over $\{A,B,C,D\}$. In particular, the points of $\bar{\mathfrak{V}}$ at infinity $(\xi_1:\xi_2:\eta_1:\eta_2:0)$ are those with $(\xi_1^2+\xi_2^2)(\eta_1^2+\eta_2^2)=0$.
Let $\cC$ be any absolutely irreducible subvariety of $\mathfrak{V}$. Then its function field $\mathbb{C}(\cC)$ is generated by $x_1,x_2,y_1,y_2$, and the above Equations (\ref{eqD071224}), (\ref{eqE071224}) and (\ref{eqG071224}) hold for $X=(x_1,x_2)$ and $Y=(y_1,y_2)$.  In particular, Equation (\ref{eqD071224}) shows that $\dim(\cC)\le 2$. This gives the following result.
\begin{lem}
\label{lem071224} The dimension of $\mathfrak{V}$ is either zero, or one.
\end{lem}
%\begin{proof} Let $\mathbb{C}(\mathfrak{V})$ be the function field associated to $\mathfrak{V}$. Since $\mathbb{C}(\mathfrak{V})$ is generated by $X_1,X_2,Y_1,Y_2$ over $\mathbb{C}$, the computation made in the proof of Lemma \ref{lemA18apr}, in particular Equation (\ref{eqD071224}), show that $\mathbb{C}(\mathfrak{V})$ is generated by $Y_1,Y_2$ over $\mathbb{C}$. Therefore, $\mathfrak{V}$ has dimension $\dim(\mathfrak{V})$ at most one.
%\end{proof}
\subsection{Case $\dim(\mathfrak{V})=1$}
In this case, $\mathfrak{V}$ is a plane curve with affine equation $f(Y_1,Y_2)=0$. Assume that System (*) has infinitely many solutions in $\mathbb{C}$. Since $\mathfrak{V}$ has finitely many absolutely irreducible components and each zero-dimensional component has finitely many points, there exists an absolutely irreducible component $\cC$ of $\mathfrak{V}$ containing infinitely many points. Let $\mathbb{C}(\cC)$ be the function field of $\cC$. Then Equations (\ref{eqE071224}) and (\ref{eqG071224}) hold for $Y=(y_1,y_2)$. Moreover, the projective closure of $\cC$ has only two points at infinity, namely $P^{+}=(1:i:0)$ and $P^{-}=(1:-i:0)$ where $i^2+1=0$.
\begin{lem}
\label{lem20apr} For $T\in\{A,B,C,D\},$ neither $\|Y-T\|^2$ nor $k_T\|Y-T\|^2-1$ vanishes at each branch of $\cC.$
\end{lem}
\begin{proof}  By hypothesis, there exists $Y_P=(\eta_1:\eta_2:1)\in \cC$ such that $(X_P,Y_P)$ is a solution of System (*). Let $\gamma$ be a branch of $\cC$ centered at $Y_P$. Since $Y_P$ is an affine point, $\gamma$ has a primitive branch representation (\ref{eqA20apr}).
Assume on the contrary $k_T\|Y-T\|^2=1$ at every branch of $\cC$. Then $k_T((x(\tau)-t_1)^2+(y(\tau)-t_2)^2)=1$ whence
$k_T((\eta_1-t_1)^2+(\eta_2-t_2)^2+\psi(\tau))=1$ where $\psi\in \mathbb{C}[[\tau]]$ with ${\rm{ord}}(\psi(\tau))>0$. Therefore, $k_T((\eta_1-t_1)^2+(\eta_2-t_2)^2)=1$, that is, $k_T(\|Y_P-T\|^2)-1=0$. But then there exists no $X_P$ such that $(X_P,Y_P)$ is a solution of System (*), a contradiction. An analog proof works for the case $\|Y-T\|^2=0$.
\end{proof}
By Lemma \ref{lem20apr}, $\Gamma_T(Y)$ is a non-zero element of $\mathbb{C}(\cC)$.
\begin{lem}
\label{lemE071224}$\Gamma_T(Y)$ does not vanishes at any branch of $\cC$ for $T\in\{A,B,C,D\}$.
\end{lem}
\begin{proof} Assume on the contrary that there exists a branch $\gamma$ of $\cC$ with center at a point $\eta_P=(\eta_1,\eta_2)$ such that $\Gamma_T(\gamma)=0$. Take a primitive branch representation of $\gamma$, say $(y_1(\tau)=\eta_1+\varphi_1(\tau),\ y_2(\tau)=\eta_2+\varphi_2(\tau))$ with $\varphi_1(\tau),\varphi_2(\tau)\in \mathbb{C}[[\tau]]$. Since $\Gamma_T(\gamma)=0$, (\ref{eq241}) yields $(y_1(\tau)-t_1)^2+(y_2(\tau)-t_2)^2=0$ in $\mathbb{C}[[\tau]$. This shows that $\gamma$ is also a branch of the reducible conic $\cC_2$ of equation $(Y_1-t_1)^2+(Y_2-t_2)^2=0$, and hence it is a branch of one of the linear components $\ell$ of $\cC_2$. Since $\cC$ is absolutely irreducible, this implies that $\cC=\ell$.
\end{proof}
\begin{lem}
\label{lem17apr} If the branch $\gamma$ of $\cC$ is centered at an affine point, then $\gamma$ is not a pole of $\Gamma_T(Y)$ with $T\in \{A,B,C,D\}$.
\end{lem}
\begin{proof} Assume on the contrary that a branch $\gamma$ is a pole of $\Gamma_A(Y)$. Let $\eta_P$ be the center.  Then, in Lemma \ref{lemB071224},
 the coefficient of $\Gamma_T(Y)$ in (\ref{eqE071224}) does not vanish as no three points from $\{A,B,C,D\}$ are collinear. Since ${\rm{ord}}(\Gamma_A(\gamma))<0$, this yields that ${\rm{ord}}(\Gamma_T(\gamma))<0$  for another $T\in\{B,C,D\}$. W.l.g., we may suppose this occurs for $\Gamma_D(Y)$.
 %Then $\gamma$ is also a pole of $\Gamma_B(Y)$. Repeat the same argument for $\Gamma_B(Y)$. It turns out that $\gamma$ is also a pole of one of the function $\Gamma_T(Y)$ with $T\in \{C,D,E,F\}$, say $T=C$.
 We show that the $\gamma$ is not a pole of any of the remaining functions $\Gamma_T(Y)$ with $T\in \{B,C\}$.
We may  assume on the contrary that $\gamma$ is a pole of each $\Gamma_T(Y)$ with $T\in\{A,B,D\}$.
Therefore $\gamma$ is also a pole of $\Gamma_T(Y)-\|T\|^2$ with $T\in \{A,B,D\}$.
Then $\gamma$ is a zero of $(\Gamma_T(Y)-\|T\|^2)^{-1}$, and hence $\|\eta_P-T\|^2-1/k_T=0$ for\ $T\in \{A,B,D\}$.
This means that $\eta_P$ is a common point of three real circles, namely those of center $T$ and radius $1/\sqrt{k_T}$ where $T\in \{A,B,D\}$.
Observe that $\eta_P$ is not a complex point, otherwise these three circles also contain the complex conjugate of $Y_P$, and hence their centers $A,B,D$ are collinear, violating assumption $(i)$. Thus, $\eta_P$ is a real point, which contradicts assumption $(ii)$. Therefore, $\gamma$ is a pole for $\Gamma_A$ and $\Gamma_D$, but it is not for $\Gamma_B$ and $\Gamma_C$.

Now, replace $(A,B,C)$ by $(B,C,D)$ in (\ref{eqG071224}). Then
\begin{equation}
\label{eqF071224}
F(Y)=
\begin{array}{l}

\Big(\left\|\begin{array}{ll}
\Gamma_C(Y)-\Gamma_B(Y)  & b_2-c_2 \\
\Gamma_D(Y)-\Gamma_B(Y) & b_2-d_2  \\
  \end{array}\right\|
  -2\Delta b_1\Big)^2+
  \vspace{0.2cm}
\Big(\left\|
\begin{array}{ll}
b_1-c_1 &\Gamma_C(Y)-\Gamma_B(Y)\\
b_1-d_1 & \Gamma_D(Y)-\Gamma_B(Y) \\
\end{array}\right\|
  -2\Delta b_2\Big)^2-\\
  4\Delta^2(\Gamma_B(Y)+\|B\|^2)=0.
\end{array}
\end{equation}
Since $\gamma$ is a pole for $\Gamma_D(Y)$ but it is not for  $\Gamma_T(Y)$ with $T\in\{B,C\}$, we have
$\Gamma_D(\gamma)=ct^{-i}+\ldots$ with $c\neq 0,$ $i>0$, whereas ${\rm{ord}}(\Gamma_T(\gamma))\ge 0$ for $T\in\{B,C\}$.
Therefore, the coefficient of $t^{-2i}$ in $F(y_1(\tau),y_2(\tau))$ must vanish where $(y_1(\tau),y_2(\tau))$ is a primitive branch representation of $\gamma$.
On the other hand, (\ref{eqF071224}) shows that in $F(Y)$ the degree of  $\Gamma_D(Y)$ is two,
and the coefficient of $t^{-2i}$ in $\Gamma_D(\gamma)^2$  is $(b_2-c_2)^2+(b_1-c_1)^2$. It turns out that $(b_2-c_2)^2+(b_1-c_1)^2=0$. Since $b_1,b_2,c_1,c_2$ are real numbers, this yields $b_1=c_1$ and $b_2=c_2$, i.e.
the points $B$ and $C$ coincide, a contradiction.
\end{proof}
\begin{lem}
\label{lemH071224} If the branch $\gamma$ is centered at a point of\, $\cC$ at infinity, then $\gamma$ is not a pole of $\Gamma_T(Y)$ for any $T\in \{A,B,C,D\}$.
\end{lem}
\begin{proof} Before Lemma \ref{lem20apr}, we have pointed out that $\cC$ has two points at infinity, namely $P^+$ and $P^-$. Let $P$ any of them. Then $P=(\eta_1:\eta_2:0)$ with $\eta_1^2+\eta_2^2=0$. Let $\gamma$ be a branch of $\cC$ centered at $P$ with a primitive branch representation
$(x(\tau):y(\tau):1)$ where
$$x(\tau)=\frac{\eta_1+\varphi_1(\tau)}{\alpha(\tau)},\quad y(\tau)=\frac{\eta_2+\varphi_2(\tau)}{\alpha(\tau)}$$
with $\alpha(\tau)=\tau^{m} \beta(\tau)$ with $m\ge 1$ and $\beta(\tau)\in \mathbb{C}[[\tau]],\  \beta(0)\neq 0$.

Suppose that  $\gamma$ is a pole of $\Gamma_A(Y)$.  Then $\gamma$ is also a pole of $\Gamma_A(Y)+\|A\|^2$, and hence $\gamma$ is a zero of $(\Gamma_A(Y)+\|A\|^2)^{-1}$, that is, $\gamma$ is a zero of
\begin{equation}
\label{eqF17apr}
k_A-\frac{1}{\|Y+A\|^2}.
\end{equation}
 By a straightforward computation,
$$
\|Y+A\|^2=(x(\tau)+a_1)^2+(y(\tau)+a_2)^2=
\frac{\sum\limits_{i=1}^2(\eta_i+\varphi_i(\tau)+a_i\alpha(\tau))^2}{\alpha(\tau)^2}=
$$

$$\frac{\sum\limits_{i=1}^2 \eta_i^2+2\sum\limits_{i=1}^2 \eta_ia_i\alpha(\tau)+2\sum\limits_{i=1}^2 a_i\alpha(\tau)\varphi_i(\tau)+\|A\|^2\alpha(\tau)^2+2\sum\limits_{i=1}^2 \eta_i\varphi_i(\tau)
+\sum\limits_{i=1}^2 \varphi_i^2(\tau)}{\alpha(\tau)^2}.
$$
By expanding (\ref{eqF17apr}),
$$k_A- \frac{ 1}{W_1(\tau)+W_2(\tau)+W_3(\tau)+Z_1(\tau)+Z_2(\tau)}$$
where
$$
W_1(\tau)=2\Big(\sum\limits_{i=1}^2 \eta_ia_i\Big)\tau^{-m}\beta(\tau)^{-1},\,
W_2(\tau)=2\Big(\sum\limits_{i=1}^2 a_i\varphi_i(\tau)\Big)\tau^{-m}\beta(\tau)^{-1},\,\,
W_3(\tau)=\|A\|^2,
$$
and

\[
\begin{array}{l}
Z_1(\tau)=2\Big(\sum\limits_{i=1}^2 \eta_i\varphi_i(\tau)\Big)\tau^{-2m}\beta(\tau)^{-2},\vspace{0.1cm}\\ 
Z_2(\tau)=\Big(\sum\limits_{i=1}^2\varphi_i(\tau)^2\Big)\tau^{-2m}\beta(\tau)^{-2}.
\end{array}
\]

If $a_1\eta_1+a_2\eta_2\neq 0$ then
$${\rm{ord}}(W_1(\tau))=-m<1-m\le{\rm{ord}}(2(a_1\varphi_1(\tau)+a_2\varphi_2(\tau)))-m\leq {\rm{ord}}(W_2(\tau)),$$ and ${\rm{ord}}(W_1(\tau))=-m<0={\rm{ord}}(W_3(\tau))$ whence ${\rm{ord}}(W_1(\tau))={\rm{ord}}(W_1(\tau))+W_2(\tau)+W_3(\tau)).$  Observe that $\gamma$ is not a pole of $(\Gamma_A(Y)+\|A\|^2)^{-1}$
otherwise $\gamma$ would be a zero of $\Gamma_A(Y)+\|A\|^2$ and hence $\gamma$ would not be a pole of $\Gamma_A(Y)$.
This implies that ${\rm{ord}}(W_1(\tau))={\rm{ord}}(Z_1(\tau)+Z_2(\tau))$. Hence
$2(a_1\eta_1+a_2\eta_2)\beta(0)^{-1}$ equals the opposite of the coefficient of $\tau^{-m}$ in $Z_1(\tau)+Z_2(\tau)$. Since the latter power series does not depend on $A$, there is a constant $\kappa$ only depending on $\gamma$ such that
\begin{equation}
\label{eqE17apr} a_1\eta_1+a_2\eta_2=\kappa.
\end{equation}
As in the proof of Lemma \ref{lem17apr}, $\gamma$ is also a pole of an other function $\Gamma_T(Y)$, say $\Gamma_D(Y)$. Since (\ref{eqE17apr}) holds true if $A$ is replaced either by $B$ or by $C$, we obtain

\begin{center}
$a_1\eta_1+a_2\eta_2=\kappa,$\\
$d_1\eta_1+d_2\eta_2=\kappa,$\\
$\eta_1^2+\eta_2^2=0.$
\end{center}

Here, $a_1d_2-d_1a_2$ may vanish. In fact, $a_1d_2-d_1a_2=0$ occurs when the points $A,D$ and the origin $O=(0,0)$ are collinear. However, this case can be avoided by a translation which takes $A$ and $B$ to two points which are not collinear with the origin.  Now, from the first two equations,
$$\eta_1=\kappa \frac{d_2-a_2}{a_1d_2-d_1a_2},\quad \eta_2=\kappa \frac{a_1-d_1}{a_1d_2-d_1a_2}.$$
Therefore,
\[
\eta_1^2+\eta_2^2=\kappa^2\frac{{((d_2-a_2)^2+(a_1-d_1)^2)}}{{(a_1d_2-d_1a_2)^2}}.
\]
Since $\eta_1^2+\eta_2^2=0$ this yields $(d_2-a_2)^2+(a_1-d_1)^2=0$ whence $A=D$ follows; but this contradicts the hypothesis that four points $A,B,C,D$ are pairwise distinct.
\end{proof}
Lemmas \ref{lem17apr} and \ref{lemH071224} have the following corollary.
\begin{lem}
\label{prop19apr} $\Gamma_T(Y)$ with $T\in \{A,B,C,D\}$ is a constant on the branches of $\cC$.
\end{lem}
From Lemma \ref{prop19apr}, there exists $c_T\in \mathbb{C}$ such that $\Gamma_T(Y)-\|T\|^2=c_T$ for all  affine points $Y\in \cC$.
Take two of them, say $Y_P$ and $Y_Q$ such that both $(X_P,Y_P)$ and $(X_Q,Y_Q)$ are real solutions of System (*) for some $X_P$ and $X_Q$. Then (\ref{eq242}) yields $\|X_P-T\|^2=c_T=\|X_Q-T\|^2$ for $T\in \{A,B,C,D\}$.
Therefore, $X_P$ and $X_Q$ are common points of four circles with centers at the points $A,B,C,D$ contradicting assumption $(ii)$.
This completes the proof of the following theorem.
\begin{thm}
\label{prop081224}
Under both assumptions $(i)$ and $(ii)$, System (*) has finitely many solutions.
\end{thm}
\subsection{Case $\dim(\mathfrak{V})=0$} In this case, $\mathfrak{V}$ has finitely many solutions. To simplify our computation, we move the point $A$ to the origin and the point $B$ to a point on the $y_1$ axis. Then we have $A=(0,0)$ and $B=(b,0)$ with $b\ne 0$. Let $\varphi$ be the dilation with center $A$ which takes $B$ to the point $(1,0)$. Here, $\varphi: (Y_1,Y_2)\mapsto (b^{-1}Y_1,b^{-1}Y_2)$. Although $\varphi$ is not a Euclidean motion, it transforms System $(*)$ to another of the same type. In fact, let $T'=\varphi(T)$ for $T\in \{A,B,C,D,X^*,Y^*\}$. Replacing the $A,B,C,D,X^*,Y^*$ in System $(*)$ with $A',B',C',D',{X^*}',{Y^*}'$, respectively, gives a system whose solutions are the same of System $(*)$ up to the scalar $b^{-1}$.  Therefore we may assume from now on that $$A=(0,0),B=(b_1,0),C=(c_1,c_2),D=(d_1,d_2),X^*=(x_1^*,x_2^*),Y^*=(y_1^*,y_2^*)\quad c_2\ne 0,d_2\ne 0.$$
Moreover, we put $\bar{\Gamma}_T=\Gamma_T+\|T\|^2$ for $T\in \{A,B,C,D\}.$
Then (\ref{eq241}), (\ref{eqE071224}) and (\ref{eqG071224}) read, respectively,
\begin{equation}
\label{eq081224A}
\begin{array}{lll}
\bar{\Gamma}_A (k_A(y_1^2+y_2^2)-1)=y_1^2+y_2^2;\\
\bar{\Gamma}_B (k_B(y_1^2+y_2^2-2y_1+1)-1)= y_1^2+y_2^2-2y_1+1;\\
\bar{\Gamma}_C (k_C(y_1^2+y_2^2-2y_1c_1-2y_2c_2+c_1^2+c_2^2)-1)= y_1^2+y_2^2-2y_1c_1-2y_2c_2+c_1^2+c_2^2;\\
\bar{\Gamma}_D (k_D(y_1^2+y_2^2-2y_1d_1-2y_2d_2+d_1^2+d_2^2)-1)=y_1^2+y_2^2-2y_1d_1-2y_2d_1+d_1^2+d_2^2.
\end{array}
\end{equation}
and
\begin{equation}
\label{eqEB1071224}
\left| \begin{array}{cccccc}
  \bar{\Gamma}_A & \bar{\Gamma}_B-1& \bar{\Gamma}_C-(c_1^2+c_2^2) & \bar{\Gamma}_D-(d_1^2+d_2^2) \\
  a_1 & b_1 & c_1 & d_1  \\
  a_2 & b_2 & c_2 & d_2  \\
  1 & 1 & 1 & 1
\end{array}\right|=0.
\end{equation}
and
\begin{equation}
\label{eqG1071224}
\begin{array}{llll}
\vspace{0.2cm}
\Big(\left|\begin{array}{lllll}
\bar{\Gamma}_B-1-\bar{\Gamma}_A  & \,\,\,\,\,0 \\
\bar{\Gamma}_C-(c_1^2+c_2^2)-\bar{\Gamma}_A & -c_2  \\
  \end{array}\right|\Big)^2+
  \vspace{0.2cm}
\Big(\left|
\begin{array}{lllll}
-1 &\bar{\Gamma}_B-1-\bar{\Gamma}_A\\
-c_1 & \bar{\Gamma}_C-(c_1^2+c_2^2)-\bar{\Gamma}_A \\
\end{array}\right|\Big)^2-
\vspace{0.2cm}
  4\Delta^2\bar{\Gamma}_A=0.
\end{array}
\end{equation}
From the first two equations in (\ref{eq081224A}),
\begin{equation}
\label{eqA151224}
y_1=\frac{1}{2}\frac{\bar{\Gamma}_A(k_B\bar{\Gamma}_B-1)-\bar{\Gamma}_B(k_A\bar{\Gamma}_A-1)+
(k_A\bar{\Gamma}_A-1)(k_B\bar{\Gamma}_B-1)}{(k_A\bar{\Gamma}_A-1)(k_B\bar{\Gamma}_B-1)}.
\end{equation}
%$$y_1=\frac{1}{2} \frac{\bar{\Gamma}_A\bar{\Gamma}_B(k_Ak_B-k_A+k_B)-\bar{\Gamma}_A(k_A+1)-\bar{\Gamma}_B(k_B-1)+1}{(k_A\bar{\Gamma}_A-1)(k_B\bar{\Gamma}_B-%1)}.$$
%$$y_1=-\frac{1}{2} \frac{k_B\bar{\Gamma}_A\bar{\Gamma}_B(k_A+1)+\bar{\Gamma}_A(k_A-1)-(\bar{\Gamma}_Bk_B+1)}%{(k_A\bar{\Gamma}_A-1)(k_B\bar{\Gamma}_B+1)}.$$
%%$$y_1=\frac{1}{2}\,\frac{k_B(k_A+1)\bar{\Gamma}_A\bar{\Gamma}_B+\bar{\Gamma}_A(k_A-1)-k_B\bar{\Gamma}_B-1}%{(k_A\bar{\Gamma}_A-1)(k_B\bar{\Gamma}_B+1)}.$$
From the first and third equations in (\ref{eq081224A}),
%\begin{equation}
\[
\label{eqL081224}
%\begin{array}{lll}
c_1y_1+c_2y_2=\frac{1}{2}\frac{\bar{\Gamma}_A(k_C\bar{\Gamma}_C-1)-\bar{\Gamma}_C(k_A\bar{\Gamma}_A-1)+
(c_1^2+c_2^2)(k_A\bar{\Gamma}_A-1)(k_C\bar{\Gamma}_C-1)}{(k_A\bar{\Gamma}_A-1)(k_C\bar{\Gamma}_C-1)}.
%-2(k_A\bar{\Gamma}_A-1)(k_C\bar{\Gamma}_C-1)(c_1y_1+c_2y_2)-\\
%\bar{\Gamma}_A+\bar{\Gamma}_C+\bar{\Gamma}_A\bar{\Gamma}_C(k_C-k_A)+(c_1^2+c_2^2)(k_A\bar{\Gamma}_A-1)%(k_C\bar{\Gamma}_C-1)=0.
%%y_1(k_Ak_Cc_1\bar{\Gamma}_A\bar{\Gamma}_C+k_Ac_1\bar{\Gamma}_A-c_1k_C\bar{\Gamma}_C-c_1)+y_2c_2(k_A\bar{\Gamma}_A-%%1)(k_C\bar{\Gamma}_C+1)\\
%(k_Ak_Cc_2\bar{\Gamma}_A\bar{\Gamma}_C+k_Ac_2\bar{\Gamma}_A-c_2k_C\bar{\Gamma}_C-c_2)+ \\
%-\ha\big(k_Ak_C(c_1^2+c_2^2)\bar{\Gamma}_A\bar{\Gamma}_C+k_C \bar{\Gamma}_A\bar{\Gamma}_C+k_A(c_1^2+c_2^2-%1)\bar{\Gamma}_A-k_C(c_1^2+c_2^2)\bar{\Gamma}_C-(c_1^2+c_2^2)\big)=0,
%\end{array}
\]
%%\end{equation}
whence
\begin{equation}
\label{eqB151224}
y_2=\frac{1}{2c_2}\Big(\frac{\bar{\Gamma}_A(k_C\bar{\Gamma}_C-1)-\bar{\Gamma}_C(k_A\bar{\Gamma}_A-1)+
(c_1^2+c_2^2)(k_A\bar{\Gamma}_A-1)(k_C\bar{\Gamma}_C-1)}{(k_A\bar{\Gamma}_A-1)(k_C\bar{\Gamma}_C-1)}
\Big)-\frac{c_1}{c_2}y_1.
\end{equation}
Substituting $y_1$ and $y_2$ in the first equation in (\ref{eq081224A}) gives a rational function whose denominator is $$4c_2^2(k_A\bar{\Gamma}_A-1)^2(k_B\bar{\Gamma}_B-1)^2(k_C\bar{\Gamma}_C-1)^2.$$ Its numerator  provides a polynomial expression $H$ of degree $6$ with a unique term of degree six namely 
%polynomial expression of the form 
%$$(k_A\bar{\Gamma}_A-1)^5(k_B\bar{\Gamma}_B+1)^2 H(\bar{\Gamma}_A,\bar{\Gamma}_B,\bar{\Gamma}_C)$$ where $H$ has degree six, and it has a unique term of degree $6$ namely $$(c_1^2+c_2^2-2c_1+1)(k_A^2c_1^2-2k_Ac_1+k_A^2c_2^2+
%1)(\bar{\Gamma}_A\bar{\Gamma}_B\bar{\Gamma}_C)^2.$$
$$ (c_1^2+c_2^2-2c_1+1)((c_1^2+c_2^2)k_A-2c_1k_A+1)\bar{\Gamma}_A^2\bar{\Gamma}_B^2\bar{\Gamma}_C^2.$$
%$$
%(-c_1 + c_1^2 + c_2^2)(2 + c_1^2 k_A + c_2^2 k_A - c_1 (2 + k_A)) k_Ak_B^2 %k_C^2\bar{\Gamma}_A^2\bar{\Gamma}_B^2\bar{\Gamma}_C^2.
%$$

To find a further equation linking $\bar{\Gamma}_A,\bar{\Gamma}_B,\bar{\Gamma}_C,\bar{\Gamma}_D$, eliminate $y_1$ from the second and third equations in (\ref{eq081224A}). The resulting polynomial equation $U(y_2,\bar{\Gamma}_A,\bar{\Gamma}_B,\bar{\Gamma}_C)=0$ is linear in $y_2$ and the coefficient of $y_2$ is equal to 
$$c_2(k_A\bar{\Gamma}_A-1)^2(k_B\bar{\Gamma}_B-1)(k_C\bar{\Gamma}_C-1).$$ 
Similarly, eliminating $y_1$ from the second and four equations gives a polynomial equation $V(y_2,\bar{\Gamma}_A,\bar{\Gamma}_B,\bar{\Gamma}_D)$ which is linear in $y_2$ with coefficient 
$$d_2(k_A\bar{\Gamma}_A-1)^2(k_B\bar{\Gamma}_B-1)(k_D\bar{\Gamma}_D-1).$$
Now, eliminating $y_2$ from these two polynomial equations we end up  
%and from the second and forth equations in (\ref{eq081224A}).  
%From the first and the and the last equations in (\ref{eq081224A}),
%\begin{equation}
%\label{eqM081224}
%\begin{array}{lll}
%y_1(k_ak_Cd_1\bar{\Gamma}_A\bar{\Gamma}_C+k_Ad_1\bar{\Gamma}_A-%d_1k_C\bar{\Gamma}_C-%d_1)+y_2(k_ak_Cd_2\bar{\Gamma}_A\bar{\Gamma}_C+k_Ad_2\bar{\Gamma}_A-%d_2k_C\bar{\Gamma}_C-d_2)+ \\
%-\ha\big(k_Ak_C(d_1^2+d_2^2)\bar{\Gamma}_A\bar{\Gamma}_C+k_C %\bar{\Gamma}_A\bar{\Gamma}_C+k_A(d_1^2+d_2^2-1)\bar{\Gamma}_A-%k_C(d_1^2+d_2^2)\bar{\Gamma}_C-(d_1^2+d_2^2)\big)=0.
%\end{array}
%\end{equation}
%Since Equation (\ref{eqM081224}) is linear in both $y_1$ and $y_2$, a %straightforward computation is enough to eliminate $y_1,y_2$ from the %above equations. The resulting equation 
with a polynomial expression of the form $(k_A\bar{\Gamma}_A-1)^3(k_B\bar{\Gamma}_B-1)$ $F(\bar{\Gamma}_A,\bar{\Gamma}_B,\bar{\Gamma}_C,\bar{\Gamma}_D)=0$, where $F$ has degree $4$ and it is linear in each indeterminate $\bar{\Gamma}_A,\bar{\Gamma}_B,\bar{\Gamma}_C,\bar{\Gamma}_D$ with a unique term of degree $4$
\begin{equation}
\label{eq091224}
%\begin{array}{lll}
%k_C\Big(k_Ak_Bk_C(\big((c_1^2+c_2^2)d_2-(d_1^2+d_2^2)c_2-c_1d_2+c_2d_1\big)+\\
%k_Ak_B(c_2-d_2)+k_Ak_C(c_1d_2-c_2d_1)+
%k_Bk_C(c_2d_1-c_1d_2+d_2-c_2)\Big)\bar{\Gamma}_A\bar{\Gamma}_B\bar{\Gamma}_C\bar{\Gamma}_D.
%\Big(k_Ak_Bk_Ck_D\big((c_1^2+c_2^2)d_2-(d_1^2+d_2^2)c_2-c_1d_2+c_2d_1\big)+k_Ak_Ck_D(c_1d_2-c_2d_1)+\\k_Bk_Ck_D(c_2d_1-c_1d_2+d_2-c_2)+k_Ak_Bk_D
%c_2-k_Ak_Bk_Cd_2\Big)\bar{\Gamma}_A\bar{\Gamma}_B\bar{\Gamma}_C\bar{\Gamma}_D.
\Big(k_Ak_Bk_Ck_D\big((c_1^2+c_2^2)d_2-(d_1^2+d_2^2)c_2-c_1d_2+c_2d_1\big)+\\k_Bk_Ck_D(c_2d_1-c_1d_2+d_2-c_2)\Big)\bar{\Gamma}_A\bar{\Gamma}_B\bar{\Gamma}_C\bar{\Gamma}_D.
%\end{array}
\end{equation}
The above equations $H(\bar{\Gamma}_A,\bar{\Gamma}_B,\bar{\Gamma}_C)=0,\  F(\bar{\Gamma}_A,\bar{\Gamma}_B,\bar{\Gamma}_C)=0$ together with (\ref{eqEB1071224}) and (\ref{eqG1071224}) may be considered as the defining equations of an algebraic set in $\mathbb{R}^4$ with coordinates $(\bar{\Gamma}_A,\bar{\Gamma}_B,\bar{\Gamma}_C,\bar{\Gamma}_D)$, and viewed as an algebraic set in $\mathbb{C}^4$. From B\'ezout's theorem, the number of its points is at most $48$.

%Now, eliminating $\bar{\Gamma}_D$ from $G(\bar{\Gamma}_A,\bar{\Gamma}_B,\bar{\Gamma}_C,\bar{\Gamma}_D)$ and (\ref{eqEB1071224}) gives a polynomial $\Phi_4$ of degree $4$  where
%$$\Phi_4= \bar{\Gamma}_A\bar{\Gamma}_B\bar{\Gamma}_C \big(\bar{\Gamma}_A + \bar{\Gamma}_B+ \bar{\Gamma}_C\big).$$

\section{Examples}
We exhibit two examples for System (*) that illustrate Theorem \ref{main}. They were found with the support of Mathematica computational system; see \cite{WR}. 
\label{exampl}
\subsubsection{First Example}
\label{seex}
We show a System (*) satisfying $(i)$ and $(ii)$ which have exactly $16$ real solutions, and $8$ complex solutions. 
Take
$$A=(-1,-1),\ B=(-1,1),\ C=(1,-1),\ D=(1,1),$$
together with $$X^*=\Big(-\sqrt{\frac{2}{11}(5-\sqrt{14})},0\Big),\quad Y^*=\Big(\sqrt{\frac{2}{11}(5-\sqrt{14})},0\Big).$$
Then $k_A=k_B=k_C=k_D=\dfrac{11}{10}$. With the support of Mathematica, we find eight solutions $(x_i,0,y_i,0)$ of System (*) where 
$$x_1=-\sqrt{\frac{2}{11}(5-\sqrt{14})},\,y_1=\sqrt{\frac{2}{11}(5-\sqrt{14})}, \qquad
x_2=\sqrt{\frac{2}{11}(5-\sqrt{14})},\,y_2=-\sqrt{\frac{2}{11}(5-\sqrt{14})},
$$
$$x_3=-\sqrt{\frac{2}{11}(5+\sqrt{14})},\,y_3=\sqrt{\frac{2}{11}(5+\sqrt{14})}, \qquad
x_4=\sqrt{\frac{2}{11}(5+\sqrt{14})},\,y_4=-\sqrt{\frac{2}{11}(5+\sqrt{14})},$$
$$x_5=\sqrt{\frac{1}{66}(85-\sqrt{2041})},\,y_5=-\sqrt{\frac{1}{66}(85+\sqrt{2041})}, \qquad 
x_6=-\sqrt{\frac{1}{66}(85-\sqrt{2041})},\,y_6=\sqrt{\frac{1}{66}(85+\sqrt{2041})},
$$
$$x_7=-\sqrt{\frac{1}{66}(85+\sqrt{2041})},\,y_7=\sqrt{\frac{1}{66}(85-\sqrt{2041})}, \qquad 
x_8=\sqrt{\frac{1}{66}(85+\sqrt{2041})},\,y_8=-\sqrt{\frac{1}{66}(85-\sqrt{2041})}.
$$
Similarly, we find eight solutions $(0,x_i,0,y_i)$  
with $x_i$ and $y_i$ as before. Furthermore, System (*) has exactly $8$ complex solutions, namely

$$\Big(\frac{1}{11}\sqrt\frac{157}{2}i,-\frac{1}{11}\sqrt\frac{107}{2}i,-\frac{1}{11}\sqrt\frac{157}{2}i,-\frac{1}{11}\sqrt\frac{107}{2}i\Big),\quad \Big(-\frac{1}{11}\sqrt\frac{157}{2}i,-\frac{1}{11}\sqrt\frac{107}{2}i,\frac{1}{11}\sqrt\frac{157}{2}i,-\frac{1}{11}\sqrt\frac{107}{2}i\Big),$$
$$\Big(\frac{1}{11}\sqrt\frac{157}{2}i,\frac{1}{11}\sqrt\frac{107}{2}i,-\frac{1}{11}\sqrt\frac{157}{2}i,\frac{1}{11}\sqrt\frac{107}{2}i\Big),\quad \Big(-\frac{1}{11}\sqrt\frac{157}{2}i,\frac{1}{11}\sqrt\frac{107}{2}i,\frac{1}{11}\sqrt\frac{157}{2}i,\frac{1}{11}\sqrt\frac{107}{2}i\Big),$$

$$\Big(-\frac{1}{11}\sqrt\frac{107}{2}i,\frac{1}{11}\sqrt\frac{157}{2}i,-\frac{1}{11}\sqrt\frac{107}{2}i,-\frac{1}{11}\sqrt\frac{157}{2}i\Big),\quad \Big(\frac{1}{11}\sqrt\frac{107}{2}i,\frac{1}{11}\sqrt\frac{157}{2}i,\frac{1}{11}\sqrt\frac{107}{2}i,-\frac{1}{11}\sqrt\frac{157}{2}i\Big),$$

$$\Big(-\frac{1}{11}\sqrt\frac{107}{2}i,-\frac{1}{11}\sqrt\frac{157}{2}i,-\frac{1}{11}\sqrt\frac{107}{2}i,\frac{1}{11}\sqrt\frac{157}{2}i\Big),\quad \Big(\frac{1}{11}\sqrt\frac{107}{2}i,-\frac{1}{11}\sqrt\frac{157}{2}i,\frac{1}{11}\sqrt\frac{107}{2}i,\frac{1}{11}\sqrt\frac{157}{2}i\Big),$$

\subsubsection{Second Example}
\label{seex2}
%We show a System (*) with infinitely many solutions which does not satisfy $(i)$. 
This time, take 
$$A=(-1,0),\ B=(1,0),\ C=(-2,0),\ D=(2,0),$$
together with 
$$X^*=(0,1),Y^*=\Big(0,\sqrt{\frac{13}{7}}i\Big).$$
Then $k_A=k_B=-\dfrac{2}{3}$ and $k_C=k_D=\dfrac{2}{3}$. We show that System (*) with these choices has infinitely many solutions. Take points $X=(0,x)$ and $Y=(0,y)$. 
Then $\|X-A\|^2=\|X-B\|^2=x^2+1$, $\|X-C\|^2=\|X-D\|^2=x^2+4$, and $\|Y-A\|^2=\|Y-B\|^2=y^2+1$, $\|Y-C\|^2=\|Y-D\|^2=y^2+4$. Then System (*) reads 
$$ \frac{1}{x^2+1} +\frac{1}{y^2+1}=-\frac{2}{3}$$
$$\frac{1}{x^2+4} +\frac{1}{y^2+4}=\frac{2}{3}.$$
Actually two these equations are the same, namely $2x^2y^2+5(x^2+y^2)+8=0$. Therefore, each of the infinitely many solutions $(x,y)$ of the latter equation defines two points $X=(0,x)$ and $Y=(0,y)$ such that $(X,Y)$ is a solution of System (*). Since $A,B,C,D$ are collinear, $(i)$ is not satisfied. On other hand,
since two of the four constants, namely $k_A$ and $k_B$ are negative, $(ii)$ holds trivially.

\end{document}